\documentclass[12pt, reqno]{amsart}
\usepackage{amsmath, amsthm, amscd, amsfonts, amssymb, graphicx, color}
\usepackage[bookmarksnumbered, colorlinks, plainpages]{hyperref}
\hypersetup{colorlinks=true,linkcolor=red, anchorcolor=green,
citecolor=cyan, urlcolor=red, filecolor=magenta, pdftoolbar=true}

\newtheorem{theorem}{Theorem}[section]
\newtheorem{lemma}[theorem]{Lemma}

\newtheorem{corollary}[theorem]{Corollary}
\theoremstyle{definition}

\newtheorem{example}[theorem]{Example}

\theoremstyle{remark}
\newtheorem{remark}[theorem]{Remark}
\numberwithin{equation}{section}

\begin{document}

\setcounter{page}{1}

\title[Norm-parallelism and the Davis--Wielandt radius]
{Norm-parallelism and the Davis--Wielandt radius of Hilbert space operators}

\author[A. Zamani, M. S. Moslehian, M.T. Chien, H. Nakazato]
{Ali Zamani$^1$, Mohammad Sal Moslehian$^2$, Mao-Ting Chien$^3$, \MakeLowercase{and} Hiroshi  Nakazato$^4$}

\address{$^1$ Department of Mathematics, Farhangian University, Iran}
\email{zamani.ali85@yahoo.com}

\address{$^2$ Department of Pure Mathematics, Center Of Excellence in Analysis on Algebraic
Structures (CEAAS), Ferdowsi University of Mashhad, P. O. Box 1159, Mashhad 91775, Iran}
\email{moslehian@um.ac.ir}

\address{$^3$ Department of Mathematics, Soochow University, Taipei 11102, Taiwan}
\email{mtchien@scu.edu.tw}

\address{$^4$ Faculty of Science and Technology, Hirosaki University, 1-bunkyocho Hirosaki-shi Aomori-ken 036-8561, Japan}
\email{nakahr@hirosaki-u.ac.jp}



\begin{abstract}
We present a necessary and sufficient
condition for the norm-parallelism of
bounded linear operators on a Hilbert space.
We also give a characterization of
the Birkhoff--James orthogonality
for Hilbert space operators.
Moreover, we discuss the connection between norm-parallelism
to the identity operator and an equality condition for
the Davis--Wielandt radius. Some other related results are also discussed.
\end{abstract} \maketitle

\section{Introduction and preliminaries}
Let $\mathbb{B}(\mathcal{H})$ denote the $C^{\ast}$-algebra of all bounded
linear operators on a complex Hilbert space $\mathcal{H}$ with an inner
product $\langle \cdot ,\cdot \rangle$ and the corresponding norm
$\|\cdot\| $. The symbol $I$ stands for the identity operator
on $\mathcal{H}$. For $T\in \mathbb{B}(\mathcal{H})$,
let $\|T\|$ and $m(T)$ denote the usual
operator norm and the minimum modulus of $T$, respectively.
Here $m(T)$ is defined to be the supremum of the set of all $\alpha\geq 0$ such that
$\|Tx\| \geq \alpha \|x\|$ for all $x \in\mathcal{H}$.
In addition, we denote by $\mathbb{M}_T$ the set of all unit
vectors at which $T$ attains its norm, i.e.,
\begin{equation*}
\mathbb{M}_T = \big\{x\in \mathcal{H}: \|x\| = 1, \|Tx\| = \|T\|\big\}.
\end{equation*}
The notion of orthogonality in $\mathbb{B}(\mathcal{H})$ can be introduced
in many ways; see \cite{B.S, M.Z} and references therein.
When $T, S\in \mathbb{B}(\mathcal{H})$, we say that $T$ is
Birkhoff--James orthogonal to $S$, and we write $T\perp_B S$, if
\begin{equation*}
\|T + \gamma S\|\geq \|T\| \qquad (\gamma \in \mathbb{C}).
\end{equation*}
This notion of orthogonality plays
a very important role in the geometry of Hilbert space operators.
Recently, some other authors studied different aspects of orthogonality of
bounded linear operators and elements of an arbitrary Hilbert $C^*$-module,
for instance, see \cite{B.S, B.G, G, M.Z, P.S.G}.

Furthermore, we say that $T\in \mathbb{B}(\mathcal{H})$ is norm-parallel
to $S\in \mathbb{B}(\mathcal{H})$ (see \cite{Z.M.1}), in short $T\parallel S$,
if there exists $\lambda\in\mathbb{T}=\big\{\alpha\in\mathbb{C}: |\alpha|=1\big\}$
such that
\begin{equation*}
\|T + \lambda S\| = \|T\| + \|S\|.
\end{equation*}
In the context of bounded linear operators,
the well-known Daugavet equation $\|T + I\| = \|T\| + 1$
is a particular case of parallelism. Such equation is a useful
property in solving a variety of problems in approximation theory;
see \cite{W, Z.M.1} and the references therein.
Some characterizations of the norm-parallelism for Hilbert space
operators and elements of an arbitrary Hilbert $C^*$-module were
given in \cite{G, W, Z.1, Z.M.1, Z.M.2}.

The numerical radius and the Crawford number of $T\in\mathbb{B}(\mathcal{H})$
are defined by
\begin{equation*}
w(T) = \sup \big\{|\langle Tx, x\rangle|:x\in\mathcal{H},\|x\| = 1\big\}
\end{equation*}
and
\begin{equation*}
c(T) = \inf \big\{|\langle Tx, x\rangle|:x\in\mathcal{H},\|x\| = 1\big\},
\end{equation*}
respectively. These concepts are useful in studying linear operators and
have attracted the attention of many authors in the last few decades
(e.g., see \cite{G.R}, and their references).
It is well known that $w(\cdot )$ defines a norm on
$\mathbb{B}(\mathcal{H})$ such that for all $T\in\mathbb{B}(\mathcal{H})$,
\begin{equation}\label{1.1}
\frac{1}{2}\|T\| \leq w(T)\leq \|T\|.
\end{equation}
The inequalities in (\ref{1.1}) are sharp. The first inequality becomes
an equality if $T^{2}=0$. The second inequality becomes an equality if $T$ is normal.
Another basic fact about the numerical radius is the power
inequality, which asserts that
\begin{equation*}
w(T^n) \leq w^n(T) \qquad (n=1, 2, \cdots)
\end{equation*}
for all $T\in \mathbb{B}(\mathcal{H})$.
For more material about the numerical radius and
other results on numerical radius inequality,
see, e.g., \cite{C.G.L, D.1, K, Y}, and the references therein.

Motivated by theoretical study and applications, there have been many
generalizations of the numerical radius; see \cite{G.R, L.Po}. One of these generalizations
is the Davis--Wielandt radius of $T\in \mathbb{B}(\mathcal{H})$ defined by
\begin{equation*}
dw(T) = \sup \big\{\sqrt{|\langle Tx, x\rangle|^2 + \|Tx\|^4}:x\in\mathcal{H},\| x\| = 1\big\};
\end{equation*}
see \cite{Da, Wi}. For $T, S\in \mathbb{B}(\mathcal{H})$ one has
\begin{itemize}
\item[(i)] $dw(T)\geq 0$ and $dw(T) = 0$ if and only if $T = 0$;
\item[(ii)] $dw(\alpha T) \begin{cases}
\geq |\alpha|dw(T) &\text{if\, $|\alpha| > 1,$}\\
= |\alpha|dw(T) &\text{if\, $|\alpha| = 1,$}\\
\leq |\alpha|dw(T) &\text{if\, $|\alpha| < 1$},
\end{cases}$
for all $\alpha \in \mathbb{C}$;
\item[(ii)] $dw(T + S) \leq \sqrt{2\big(dw(T) + dw(S)\big) + 4\big(dw(T) + dw(S)\big)^2}$,
\end{itemize}
and therefore $dw(\cdot)$ cannot be a norm on $\mathbb{B}(\mathcal{H})$.
In spite of this, it has many interesting properties.
The following property of $dw(\cdot)$ is immediate:
\begin{equation}\label{I.1.2}
\max\big\{w(T), \|T\|^2\big\} \leq dw(T) \leq \sqrt{w^2(T) + \|T\|^4}.
\end{equation}
We remark that the upper bound and lower bound in (\ref{I.1.2}) are both attainable.
In fact, if $T =
\begin{bmatrix}
1 & 0 \\
0 & 0
\end{bmatrix}$ and $S =
\begin{bmatrix}
0 & 1 \\
0 & 0
\end{bmatrix},$
then simple computations show that
$dw(T) = \sqrt{w^2(T) + \|T\|^4} = \sqrt{2}$ and
$\max\big\{w(S), \|S\|^2\big\} = dw(S) = 1$.

This paper is organized as follows. In the second section,
we present a necessary and sufficient condition
for $T\in \mathbb{B}(\mathcal{H})$ to be norm-parallel to
$S\in \mathbb{B}(\mathcal{H})$. We also give a characterization of
the Birkhoff--James orthogonality in $\mathbb{B}(\mathcal{H})$
for Hilbert space operators.
Moreover, we obtain some new refinements of numerical radius
inequalities for Hilbert space operators.
In Section 3, the relation of the norm-parallelism of operators
and their Davis--Wielandt radii is discussed. In particular,
we show that $T\parallel I$ if and only if
$dw(T) = \sqrt{w^2(T) + \|T\|^4}$.
Some other related results are also presented.

\section{Characterization of norm-parallelism and Birkhoff--James Orthogonality of operators}
We begin with a useful lemma which we will use frequently in the present paper.
\begin{lemma}\label{L.2.1}
Let $\mathcal{H}$ be a Hilbert space and $a, b \in \mathcal{H}$. Then
\begin{equation*}
\|a + \gamma b\|^2\|b\|^2 - \big|\langle a + \gamma b, b\rangle\big|^2
= \|a\|^2\|b\|^2 - \big|\langle a, b\rangle\big|^2 \qquad (\gamma \in \mathbb{C}).
\end{equation*}
\end{lemma}
\begin{proof}
The proof is straightforward so we omit it.
\end{proof}
\begin{remark}
According to the above lemma, we have
\begin{align}\label{I.2.0}
\|b\|^2\inf_{\gamma \in \mathbb{C}}\|a + \gamma b\|^2
= \|a\|^2\,\|b\|^2 - |\langle a, b\rangle|^2 \qquad (a, b\in \mathcal{H})
\end{align}
which is a well-known identity. In particular, $a$ and $b$ are linearly dependent
if and only if $\inf_{\gamma \in \mathbb{C}}\|a + \gamma b\|^2 = 0$, that is,
if and only if $|\langle a, b\rangle| = \|a\|\,\|b\|$.
Further, The equality (\ref{I.2.0}) shows that two elements $a$ and $b$ of a Hilbert space
are orthogonal in the sense of the inner product precisely when they are the
Birkhoff--James orthogonal, that is, $\inf_{\gamma \in \mathbb{C}}\|a + \gamma b\| = \|a\|$.
\end{remark}
For elements $T, S\in\mathbb{B}(\mathcal{H})$,
it was proved in \cite[Theorem 3.3]{Z.M.1}
that $T\parallel S$ if and only if there exists a sequence of unit
vectors $\{x_n\}$ in $\mathcal{H}$ such that
\begin{align}\label{I.3.01}
\lim_{n\rightarrow\infty} \big|\langle Tx_n, Sx_n\rangle\big| = \|T\|\|S\|.
\end{align}
It follows then that if the Hilbert space $\mathcal{H}$ is finite dimensional,
$T\parallel S$ if and only if there exists a unit vector $x\in \mathcal{H}$ such that
$\big|\langle Tx, Sx\rangle\big| = \|T\|\,\|S\|$.
Notice that the condition of finite dimensionality is essential (see \cite[Example 2.17]{Z.M.2}).
In addition, for compact operators $T, S$ on a Hilbert space $\mathcal{H}$
(not necessarily finite dimensional)
it was proved in \cite[Theorem 2.10]{Z.1} that $T\parallel S$ if and only if
there exists $x\in \mathbb{M}_T\cap \mathbb{M}_S$ such that
\begin{align}\label{I.3.02}
\big|\langle Tx, Sx\rangle\big| = \|T\|\|S\|.
\end{align}
Next we obtain a necessary and sufficient condition
for $T\in \mathbb{B}(\mathcal{H})$ to be norm-parallel to $S\in \mathbb{B}(\mathcal{H})$.
\begin{theorem}\label{T.2.1}
Let $\mathcal{H}$ be a Hilbert space and $T, S\in \mathbb{B}(\mathcal{H})$
be compact operators. Then the following statements are equivalent:
\begin{itemize}
\item[(i)] $T\parallel S$.
\item[(ii)] There exists $x\in \mathbb{M}_T\cap \mathbb{M}_S$ such that for every
$\gamma \in \mathbb{C}$ the vectors $Tx + \gamma Sx$ and $Sx$
are linearly dependent.
\end{itemize}
\end{theorem}
\begin{proof}
(i)$\Rightarrow$(ii) Let $T\parallel S$.
From (\ref{I.3.02}), there exists $x\in \mathbb{M}_T\cap \mathbb{M}_S$ such that
$\big|\langle Tx, Sx\rangle\big| = \|T\|\|S\|$.
Choose $a = Tx$ and $b = Sx$ in Lemma \ref{L.2.1} to get
\begin{align}\label{I.2.2}
\|Tx + \gamma Sx\|^2\|Sx\|^2 - \big|\langle Tx + \gamma Sx, Sx\rangle\big|^2
= \|Tx\|^2\|Sx\|^2 - \big|\langle Tx, Sx\rangle\big|^2,
\end{align}
for all $\gamma \in \mathbb{C}$. It follows from (\ref{I.2.2}) that
\begin{equation*}
\|Tx + \gamma Sx\|^2\|Sx\|^2 - \big|\langle Tx + \gamma Sx, Sx\rangle\big|^2 = 0,
\end{equation*}
or equivalently,
\begin{equation*}
\big|\langle Tx + \gamma Sx, Sx\rangle\big|
= \|Tx + \gamma Sx\|\,\|Sx\| \qquad (\gamma \in \mathbb{C}).
\end{equation*}
By the condition for equality in the Cauchy--Schwarz inequality
we conclude that the vectors $Tx + \gamma Sx$ and $Sx$ are linearly dependent.

The implication (ii)$\Rightarrow$(i) follows also by the same argument.
\end{proof}
For $T, S\in \mathbb{B}(\mathcal{H})$, Bhatia and \v{S}emrl \cite[Theorem 1.1 and Remark 3.1]{B.S}
proved that $T\perp_B S$ if and only if there exists
a sequence of unit vectors $\{x_n\}$ in $\mathcal{H}$ such that
\begin{align*}
\lim_{n\rightarrow\infty} \|Tx_n\| = \|T\| \quad
\mbox{and} \quad \lim_{n\rightarrow\infty}\langle Tx_n, Sx_n\rangle = 0.
\end{align*}
For a compact operator $T$ on $\mathcal{H}$ and $S\in \mathbb{B}(\mathcal{H})$
it was shown in \cite[Corollary 2.6]{T} that $T\perp_B S$ if and only if
there exists $x\in \mathbb{M}_T$ such that
\begin{align}\label{I.2.21}
\langle Tx, Sx\rangle = 0.
\end{align}
In the following, we give a characterization of the Birkhoff--James orthogonality
for operators in $\mathbb{B}(\mathcal{H})$.
\begin{theorem}\label{T.2.2}
Let $T, S\in \mathbb{B}(\mathcal{H})$ and suppose that $T$ is compact.
Then the following statements are equivalent:
\begin{itemize}
\item[(i)] $T\perp_B S$.
\item[(ii)] There exists $x\in \mathbb{M}_T$
such that for every $\gamma \in \mathbb{C}$
\begin{equation*}
\|Tx + \gamma Sx\|^2 = \|Tx\|^2 + |\gamma|^2\|Sx\|^2.
\end{equation*}
\end{itemize}
\end{theorem}
\begin{proof}
(i)$\Rightarrow$(ii) Let $T\perp_B S$.
By (\ref{I.2.21}), there exists $x\in \mathbb{M}_T$ such that
$\langle Tx, Sx\rangle = 0$.
As in the proof of Theorem \ref{T.2.1}, equality (\ref{I.2.2}) gives
\begin{equation*}
\|Tx + \gamma Sx\|^2\|Sx\|^2 - |\gamma|^2\|Sx\|^4
= \|Tx\|^2\|Sx\|^2,
\end{equation*}
and hence
\begin{equation*}
\|Tx + \gamma Sx\|^2 = \|Tx\|^2 + |\gamma|^2\|Sx\|^2 \qquad (\gamma \in \mathbb{C}).
\end{equation*}
The same proof works for the implication (ii)$\Rightarrow$(i).
\end{proof}
Recently, Turn\v{s}ek \cite{T} introduced a weaker notions
of operator Birkhoff–-James orthogonality.
It is said that $T\in \mathbb{B}(\mathcal{H})$ is $r$-orthogonal
to $S\in \mathbb{B}(\mathcal{H})$, denoted by $T\perp^r_B S$ if
$\|T + \gamma S\|\geq \|T\|$ for all $\gamma \in \mathbb{R}$.
Some applications of the $r$-orthogonality in the geometry of Hilbert space
operators can be found in \cite{T}.
In \cite[Corollary 2.6]{T} (see also \cite[Theorem 2.6]{B.G})
the author proved that for compact operator $T$ on $\mathcal{H}$ and
$S\in \mathbb{B}(\mathcal{H})$, we have $T\perp^r_B S$
if and only if there exists $x\in \mathbb{M}_T$ such that
\begin{align}\label{I.2.22}
\mbox{Re}\langle Tx, Sx\rangle = 0.
\end{align}
We next give a characterization of
$r$-orthogonality of operators in $\mathbb{B}(\mathcal{H})$.
\begin{theorem}
Let $T, S\in \mathbb{B}(\mathcal{H})$ and suppose that $T$ is compact.
Then the following statements are equivalent:
\begin{itemize}
\item[(i)] $T\perp^r_B S$.
\item[(ii)] There exists $x\in \mathbb{M}_T$
such that for every $\gamma \in \mathbb{R}$
\begin{equation*}
\|Tx + \gamma Sx\|^2 = \|Tx\|^2 + \gamma^2\|Sx\|^2.
\end{equation*}
\end{itemize}
\end{theorem}
\begin{proof}
Let $a, b \in \mathcal{H}$. A direct calculation shows that
\begin{align}\label{I.2.23}
\|a + \gamma b\|^2\|b\|^2 - \big(\mbox{Re}\langle a + \gamma b, b\rangle\big)^2
= \|a\|^2\|b\|^2 - \big(\mbox{Re}\langle a, b\rangle\big)^2
\qquad (\gamma \in \mathbb{R}).
\end{align}
As in the proof of Theorem \ref{T.2.2},
the equivalence (i)$\Leftrightarrow$(ii)
follows from (\ref{I.2.22}) and (\ref{I.2.23}).
The details are left to the reader.
\end{proof}
We finish this section with some new refinements of numerical radius
inequalities for Hilbert space operators.
The following auxiliary results are needed.
\begin{lemma}\cite{D.1}\label{I.2.0000}
If $a, b, e \in \mathcal{H}$ and $\|e\| = 1$, then
\begin{equation*}
2\big|\langle a, e\rangle\langle e, b\rangle\big|
\leq \|a\|\|b\| + |\langle a, b\rangle|.
\end{equation*}
\end{lemma}
\begin{lemma}\cite[Corollary 2.3]{Z.2}\label{I.2.0001}
Let $T\in \mathbb{B}(\mathcal{H})$. Then
\begin{equation*}
\|T\|^2 + c^2(T) \leq \|Tx\|^2 + |\langle Tx, x\rangle|^2
\leq 4w^2(T) \qquad (x\in \mathcal{H}, \|x\| = 1).
\end{equation*}
\end{lemma}
\begin{theorem}
Let $T\in \mathbb{B}(\mathcal{H})$ and $\xi \in \mathbb{C}-\{0\}$. Then
\begin{itemize}
\item[(i)] $\|T\|^2 - w^2(T) \leq
\inf_{\gamma \in \mathbb{C}}\Big\{\|T - \gamma I\|^2 - c^2(T - \gamma I)\Big\}$.
\item[(ii)] $\left(1 - \frac{\|T - \xi I\|^2}{|\xi|^2}\right)\|T\|^2
\leq w^2(T) - \frac{c^2(T^*T - \xi T^*)}{|\xi|^2}$.
\item[(iii)] $w^2(T) - w(T^2) \leq \inf_{\gamma \in \mathbb{C}}
\Big\{\frac{\|T - \gamma T^*\|^2\|T\|^2 - c^2(T^2 - \gamma TT^*)}{\|T\|^2 + c^2(T)}\Big\}$.
\end{itemize}
\end{theorem}
\begin{proof}
(i) Suppose that $x\in \mathcal{H}$ with $\|x\|= 1$. Choose $a = T x$ and $b = -x$
in Lemma \ref{L.2.1} to give
\begin{align*}
\|Tx\|^2 - |\langle Tx, x\rangle|^2 =
\|Tx - \gamma x\|^2 - |\langle Tx - \gamma x, x\rangle|^2 \qquad (\gamma \in \mathbb{C}),
\end{align*}
whence
\begin{align*}
\|Tx\|^2 - |\langle Tx, x\rangle|^2 \leq
\|T - \gamma I\|^2 - c^2(T - \gamma I) \qquad (\gamma \in \mathbb{C}).
\end{align*}
Thus
\begin{align}\label{I.4.1}
\|Tx\|^2 - |\langle Tx, x\rangle|^2 \leq
\inf_{\gamma \in \mathbb{C}}\Big\{\|T - \gamma I\|^2 - c^2(T - \gamma I)\Big\}.
\end{align}
Now, if we take the supremum over all $x\in \mathcal{H}$ with $\|x\|= 1$
in (\ref{I.4.1}), then we get
\begin{align*}
\|T\|^2 - w^2(T) \leq
\inf_{\gamma \in \mathbb{C}}\Big\{\|T - \gamma I\|^2 - c^2(T - \gamma I)\Big\}.
\end{align*}

(ii) Put $a = -\xi x$, $b = Tx$ and $\gamma = 1$,
where $x \in \mathcal{H}$, $\|x\| = 1$, in Lemma \ref{L.2.1}. We get
\begin{align*}
|\xi|^2\|Tx\|^2 - |\xi|^2\big|\langle x, Tx\rangle\big|^2
= \|-\xi x + Tx\|^2\|Tx\|^2 - \big|\langle -\xi x + Tx, Tx\rangle\big|^2.
\end{align*}
Thus
\begin{align*}
\|Tx\|^2 - \big|\langle Tx, x\rangle\big|^2
= \|Tx\|^2\frac{\|Tx - \xi x\|^2}{|\xi|^2} -
\frac{\big|\langle T^*Tx - \xi T^*x, x\rangle\big|^2}{|\xi|^2},
\end{align*}
which implies
\begin{align*}
\|Tx\|^2 - \big|\langle Tx, x\rangle\big|^2
\leq \|T\|^2\frac{\|T - \xi I\|^2}{|\xi|^2} - \frac{c^2(T^*T - \xi T^*)}{|\xi|^2}.
\end{align*}
Taking the supremum over all $x \in \mathcal{H}$, $\|x\| = 1$, we deduce
\begin{align*}
\|T\|^2 - w^2(T)
\leq \|T\|^2\frac{\|T - \xi I\|^2}{|\xi|^2} - \frac{c^2(T^*T - \xi T^*)}{|\xi|^2},
\end{align*}
and then also
$\left(1 - \frac{\|T - \xi I\|^2}{|\xi|^2}\right)\|T\|^2
\leq w^2(T) - \frac{c^2(T^*T - \xi T^*)}{|\xi|^2}$.

(iii) Put $a = T x$ and $b = -T^*x$, where $x \in \mathcal{H}$,
$\|x\| = 1$, in Lemma \ref{L.2.1}.
For any $\gamma \in \mathbb{C}$, we deduce that
\begin{align*}
\|Tx\|^2\|T^*x\|^2 - \big|\langle Tx, T^*x\rangle\big|^2
= \|Tx - \gamma T^*x\|^2\|T^*x\|^2 - \big|\langle Tx - \gamma T^*x, T^*x\rangle\big|^2,
\end{align*}
that is
\begin{align*}
\|Tx\|^2\|T^*x\|^2 = \big|\langle T^2x, x\rangle\big|^2
+ \|Tx - \gamma T^*x\|^2\|T^*x\|^2 - \big|\langle T^2x - \gamma TT^*x, x\rangle\big|^2.
\end{align*}
This ensures
\begin{align}\label{I.10005}
\|Tx\|^2\|T^*x\|^2 \leq w^2(T^2) + \|T - \gamma T^*\|^2\|T\|^2 - c^2(T^2 - \gamma TT^*).
\end{align}
Now, putting $a = Tx$, $b = T^*x$ and $e = x$ in Lemma \ref{I.2.0000} gives
\begin{align*}
2\big|\langle Tx, x\rangle\langle x, T^*x\rangle\big|
\leq \|Tx\|\|T^*x\| + |\langle Tx, T^*x\rangle|,
\end{align*}
whence
\begin{align*}
2|\langle Tx, x\rangle|^2 \leq \|Tx\|\|T^*x\| + |\langle T^2x, x\rangle|.
\end{align*}
Thus
\begin{align}\label{I.10006}
2|\langle Tx, x\rangle|^2 \leq \|Tx\|\|T^*x\| + w(T^2).
\end{align}
By (\ref{I.10005}) and (\ref{I.10006}), we obtain
\begin{align*}
2|\langle Tx, x\rangle|^2 \leq \sqrt{w^2(T^2)
+ \|T - \gamma T^*\|^2\|T\|^2 - c^2(T^2 - \gamma TT^*)} + w(T^2).
\end{align*}
Taking the supremum in the above inequality over all $x \in \mathcal{H}$, $\|x\| = 1$, we get
\begin{align*}
2w^2(T) \leq \sqrt{w^2(T^2) + \|T - \gamma T^*\|^2\|T\|^2 - c^2(T^2 - \gamma TT^*)} + w(T^2),
\end{align*}
or equivalently,
\begin{align*}
(2w^2(T) - w(T^2))^2\leq w^2(T^2) + \|T - \gamma T^*\|^2\|T\|^2 - c^2(T^2 - \gamma TT^*).
\end{align*}
This yields
\begin{align*}
w^2(T) - w(T^2) \leq
\frac{\|T - \gamma T^*\|^2\|T\|^2 - c^2(T^2 - \gamma TT^*)}{4w^2(T)}.
\end{align*}
Furthermore, by Lemma \ref{I.2.0001} we have $\|T\|^2 + c^2(T) \leq 4w^2(T)$.
So, from the above inequality we get
\begin{align*}
w^2(T) - w(T^2) \leq
\frac{\|T - \gamma T^*\|^2\|T\|^2 - c^2(T^2 - \gamma TT^*)}{\|T\|^2 + c^2(T)}
\qquad (\gamma \in \mathbb{C}).
\end{align*}
Finally we conclude that
\begin{align*}
w^2(T) - w(T^2) \leq
\inf_{\gamma \in \mathbb{C}}
\Big\{\frac{\|T - \gamma T^*\|^2\|T\|^2 - c^2(T^2 - \gamma TT^*)}{\|T\|^2 + c^2(T)}\Big\}.
\end{align*}
\end{proof}
\section{Norm-parallelism of operators and an equality condition for the Davis--Wielandt radius}
To establish the following result we use
some ideas of \cite[Lemma 1]{Li}.
\begin{theorem}\label{T.3.1}
Let $T\in \mathbb{B}(\mathcal{H})$.
Then the following conditions are equivalent:
\begin{itemize}
\item[(i)] $T \parallel I$.
\item[(ii)] $dw(T) = \sqrt{w^2(T) + \|T\|^4}$.
\end{itemize}
\end{theorem}
\begin{proof}
(i)$\Rightarrow$(ii) Let $T\parallel I$.
From (\ref{I.3.01}), there exists a sequence of unit
vectors $\{x_n\}$ in $\mathcal{H}$ such that
\begin{align}\label{I.3.1}
\lim_{n\rightarrow \infty } \big|\langle Tx_n, x_n\rangle\big| = \|T\|.
\end{align}
We have
\begin{align}\label{I.3.2}
\big|\langle Tx_n, x_n\rangle\big|\leq \|Tx_n\| \leq \|T\| \qquad
\mbox{and} \qquad \big|\langle Tx_n, x_n\rangle\big|\leq w(T) \leq \|T\|.
\end{align}
Hence, by (\ref{I.3.1}) and (\ref{I.3.2}), we obtain
\begin{align}\label{I.3.3}
\lim_{n\rightarrow \infty } \|Tx_n\| = \|T\| \qquad
\mbox{and} \qquad \lim_{n\rightarrow \infty } \big|\langle Tx_n, x_n\rangle\big| = w(T).
\end{align}
Also, by the definition of $dw(T)$ we have
\begin{align}\label{I.3.4}
\sqrt{\big|\langle Tx_n, x_n\rangle\big|^2 + \|Tx_n\|^4}\leq dw(T) \leq \sqrt{w^2(T) + \|T\|^4}
\end{align}
It follows from (\ref{I.3.3}) and (\ref{I.3.4}) that
\begin{align*}
dw(T) = \sqrt{w^2(T) + \|T\|^4}
\end{align*}
(ii)$\Rightarrow$(i) Let $dw(T) = \sqrt{w^2(T) + \|T\|^4}$.
So, by the definition of $dw(T)$, there exists a sequence of unit
vectors $\{x_n\}$ in $\mathcal{H}$ such that
\begin{align*}
\lim_{n\rightarrow \infty }\sqrt{\big|\langle Tx_n, x_n\rangle\big|^2 + \|Tx_n\|^4}
= \sqrt{w^2(T) + \|T\|^4}.
\end{align*}
Then we have
\begin{align}\label{I.3.6}
\lim_{n\rightarrow \infty } \big|\langle Tx_n, x_n\rangle\big| = w(T) \qquad
\mbox{and} \qquad \lim_{n\rightarrow \infty } \|Tx_n\| = \|T\|.
\end{align}
Our aim is to show that $w(T) = \|T\|$, and hence by (\ref{I.3.6}) we obtain
\begin{align*}
\lim_{n\rightarrow \infty } \big|\langle Tx_n, x_n\rangle\big| = \|T\|,
\end{align*}
or equivalently $T \parallel I$.
Write
\begin{align}\label{I.3.8}
Tx_n = \alpha_n x_n + \beta_n y_n
\end{align}
with $\langle x_n, y_n\rangle = 0$, $\|y_n\| = 1$ and
$\alpha_n, \beta_n \in \mathbb{C}$.
It follows from (\ref{I.3.6}) and (\ref{I.3.8}) that
$\alpha_n = \langle Tx_n, x_n\rangle$, $\beta_n = \langle Tx_n, y_n\rangle$,
$\lim_{n\rightarrow \infty} |\alpha_n| = w(T)$
and $$\lim_{n\rightarrow \infty} |\alpha_n|^2 + |\beta_n|^2 = \|T\|^2.$$
Put $\gamma_n = \langle Ty_n, x_n\rangle$, $\delta_n = \langle Ty_n, y_n\rangle$ and
$T_n = \begin{bmatrix}
\alpha_n & \gamma_n \\
\beta_n & \delta_n
\end{bmatrix}.$
Since
\begin{align*}
|\alpha_n| \leq w(T_n) \leq w(T)
\end{align*}
then
\begin{align}\label{I.3.9}
\lim_{n\rightarrow \infty} w(T_n) = w(T).
\end{align}
Furthermore, we have
\begin{align*}
|\alpha_n|^2 \leq w\left(\begin{bmatrix}
|\alpha_n|^2 &
\frac{\overline{\alpha}_n\gamma_n + \alpha_n \overline{\beta}_n}{2} \\
\frac{\overline{\alpha}_n\beta_n + \alpha_n \overline{\gamma}_n}{2} &
\frac{\overline{\alpha}_n\delta_n + \alpha_n \overline{\delta}_n}{2}
\end{bmatrix}\right)
= w\big(\mbox{Re}(\overline{\alpha}_nT_n \big)
\leq w(\overline{\alpha}_nT_n) \leq w^2(T).
\end{align*}
Thus $\lim_{n\rightarrow \infty} w\big(\mbox{Re}(\overline{\alpha}_nT_n \big) = w^2(T)$
and $\lim_{n\rightarrow \infty} \frac{\overline{\alpha}_n\gamma_n + \alpha_n \overline{\beta}_n}{2} =0$.
It follows that
\begin{align}\label{I.3.10}
\lim_{n\rightarrow \infty} |\gamma_n| = \lim_{n\rightarrow \infty} |\beta_n|.
\end{align}
On the other hands, we have
\begin{align*}
T^*_nT_n = \begin{bmatrix}|\alpha_n|^2 + |\beta_n|^2 &
\overline{\alpha}_n\gamma_n + \overline{\beta}_n\delta_n \\
\alpha_n\overline{\gamma}_n + \beta_n\overline{\delta}_n &
|\gamma_n|^2 + |\delta_n|^2
\end{bmatrix}.
\end{align*}
Therefore, we obtain
\begin{align*}
|\alpha_n|^2 + |\beta_n|^2
\leq \|T^*_nT_n\| = \|T_n\|^2 \leq \|T\|^2.
\end{align*}
It follows from the above inequality that
$\lim_{n\rightarrow \infty} \|T^*_nT_n\| = \|T\|^2$ and hence we get
$\lim_{n\rightarrow \infty} \overline{\alpha}_n\gamma_n + \overline{\beta}_n\delta_n =0$.
This yields
\begin{align}\label{I.3.11}
\lim_{n\rightarrow \infty} |\delta_n| = \lim_{n\rightarrow \infty} |\alpha_n|.
\end{align}
By (\ref{I.3.10}) and (\ref{I.3.11}) we reach
\begin{align*}
\lim_{n\rightarrow \infty} |\gamma_n|^2 + |\delta_n|^2
= \lim_{n\rightarrow \infty}|\alpha_n|^2 + |\beta_n|^2
= \|T\|^2,
\end{align*}
from which we get
\begin{align*}
\lim_{n\rightarrow \infty} T^*_nT_n =
\begin{bmatrix}\|T\|^2 & 0 \\
0 & \|T\|^2
\end{bmatrix}.
\end{align*}
It follows that
\begin{align}\label{I.3.12}
\lim_{n\rightarrow \infty} w(T_n) = \|T\|.
\end{align}
From (\ref{I.3.9}) and (\ref{I.3.12}) we conclude that $w(T) = \|T\|$.
\end{proof}
As a consequence of Theorem \ref{T.3.1}, we have the following result.
\begin{corollary}\label{C.3.1}
Let $T\in \mathbb{B}(\mathcal{H})$.
Then the following conditions are equivalent:
\begin{itemize}
\item[(i)] $dw(T) = \sqrt{w^2(T) + \|T\|^4}$.
\item[(ii)] $w(T) = \|T\|$.
\item[(iii)] $dw(T) = \|T\|\sqrt{1 + \|T\|^2}$.
\item[(iv)] $T^*T \leq w^2(T)I$.
\end{itemize}
\end{corollary}
\begin{proof}
The equivalence (i)$\Leftrightarrow$(ii)
follows from the proof of Theorem \ref{T.3.1}.

(i)$\Rightarrow$(iii) This implication follows from
the equivalence (i)$\Leftrightarrow$(ii).

(iii)$\Rightarrow$(i) Let (iii) holds.
Since $w(T) \leq \|T\|$, we have
\begin{align*}
\|T\|\sqrt{1 + \|T\|^2} = dw(T) \leq
\sqrt{w^2(T) + \|T\|^4} \leq \|T\|\sqrt{1 + \|T\|^2},
\end{align*}
and so $dw(T) = \sqrt{w^2(T) + \|T\|^4}$.

(i)$\Leftrightarrow$(iv) By the equivalence (i)$\Leftrightarrow$(ii),
$dw(T) = \sqrt{w^2(T) + \|T\|^4}$ if and only if
$w(T) = \|T\|$, that is, $\|Tx\| \leq w(T)\|x\|$ for all $x\in \mathcal{H}$.
This is equivalent to $\|Tx\|^2 \leq w^2(T)\|x\|^2$ for all $x\in \mathcal{H}$,
that is, $\langle Tx, Tx\rangle \leq \langle w^2(T)x, x\rangle$
for all $x\in \mathcal{H}$, and finally $\langle T^*Tx - w^2(T)x, x\rangle \leq 0$
for all $x\in \mathcal{H}$, or equivalently, $T^*T \leq w^2(T)I$.
\end{proof}
Recall that if $x, y\in \mathcal{H}$, then $\|x\otimes y\| = \|x\|\|y\|$
and $w(x\otimes y) = \frac{1}{2}\left(|\langle x, y\rangle| + \| x\| \| y\|\right)$,
where $x\otimes y$ is the rank one operator defined by
$(x\otimes y)(z) : = \langle z, y\rangle x$ for all $z\in \mathcal{H}$.
\begin{corollary}
For $x, y \in \mathcal{H}$, the following conditions are equivalent:
\begin{itemize}
\item[(ii)] $dw(x\otimes y) = \sqrt{w^2(x\otimes y) + \|x\otimes y\|^4}$.
\item[(ii)] The vectors $x$ and $y$ are linearly dependent.
\end{itemize}
\end{corollary}
\begin{proof}
This follows form the equivalence (i)$\Leftrightarrow$(ii) of Corollary \ref{C.3.1}.
\end{proof}
For $T\in \mathbb{B}(\mathcal{H})$ the following results were
obtained in \cite{Z.M.1}:
\begin{align}\label{I.3.14}
T \parallel I \Leftrightarrow T\parallel T^* \qquad
\mbox{and} \qquad T \parallel I \Leftrightarrow T^*T\parallel T^*.
\end{align}
As an immediate consequence of (\ref{I.3.01}), (\ref{I.3.14})
and Theorem \ref{T.3.1}, we have the following result.
\begin{corollary}
Let $T\in \mathbb{B}(\mathcal{H})$. The following statements are equivalent:
\begin{itemize}
\item[(i)] $dw(T) = \sqrt{w^2(T) + \|T\|^4}$.
\item[(ii)] There exists a sequence of unit vectors $\{x_n\}$ in $\mathcal{H}$ such that
\begin{equation*}
\lim_{n\rightarrow\infty} \big|\langle T^2x_n, x_n\rangle\big| = \|T\|^2.
\end{equation*}
\item[(iii)] There exists a sequence of unit vectors $\{x_n\}$ in $\mathcal{H}$ such that
\begin{equation*}
\lim_{n\rightarrow\infty} \big|\langle TT^*Tx_n, x_n\rangle\big| = \|T\|^3.
\end{equation*}
\end{itemize}
\end{corollary}
\begin{corollary}\label{Cr.3.10}
Let $\mathcal{H}$ be a finite dimensional Hilbert space and
$T\in \mathbb{B}(\mathcal{H})$. The following statements are equivalent:
\begin{itemize}
\item[(i)] $dw(T) = \sqrt{w^2(T) + \|T\|^4}$.
\item[(ii)] There exists a unit vector $x\in \mathcal{H}$
such that $|\langle Tx, x\rangle| = \|T\|$.
\item[(iii)] There exists $x\in \mathbb{M}_T$
such that for every $\gamma \in \mathbb{C}$
the vectors $Tx + \gamma x$ and $x$ are linearly dependent.
\end{itemize}
\end{corollary}
\begin{proof}
This follows immediately form (\ref{I.3.02}) and Theorems \ref{T.3.1}, \ref{T.2.1}.
\end{proof}
The following example shows that the condition of finite dimensionality in the
implication (i)$\Rightarrow $(ii) of Corollary \ref{Cr.3.10} is essential.
\begin{example}
Consider the shift operator $T\,:\ell^2\longrightarrow \ell^2$ defined by
$$T(\xi_1, \xi_2, \xi_3, \cdots) = (0, \xi_1, \xi_2, \xi_3, \cdots).$$
One can easily observe that $dw(T) = \sqrt{w^2(T) + \|T\|^4} = \sqrt{2}$,
but there is no unit vector $x\in\ell^2$ such that $|\langle Tx, x\rangle| = \|T\|.$
\end{example}
For a subspace $\mathcal{H}_0$ of a Hilbert space $\mathcal{H}$ let
$\mathbb{S}_{\mathcal{H}_0} = \{x\in \mathcal{H}_0:\, \|x\| = 1\}$ and
$\|T\|_{{\mathcal{H}_0}^\perp} = \sup\big\{\|Tx\|: x\in {\mathcal{H}_0}^\perp, \, \|x\| = 1\big\}$.
Let us quote a result from \cite{Z.M.2}.
\begin{lemma}\label{L.3.1}
Let $T\in\mathbb{B}(\mathcal{H})$. If $\mathbb{S}_{\mathcal{H}_0} = \mathbb{M}_T$,
where $\mathcal{H}_0$ is a finite dimensional subspace of $\mathcal{H}$
and $\|T\|_{{\mathcal{H}_0}^\perp}< \|T\|$. Then for any
$S\in\mathbb{B}(\mathcal{H})$ the following statements are equivalent:
\begin{itemize}
\item[(i)] $T\parallel S$.
\item[(ii)] There exists a unit vector $x\in \mathcal{H}_0$
such that $|\langle Tx, Sx\rangle| = \|T\|\|S\|$.
\end{itemize}
\end{lemma}
As an immediate consequence of Theorem \ref{T.3.1} and Lemma \ref{L.3.1},
we have the following result.
\begin{corollary}
Let $T\in\mathbb{B}(\mathcal{H})$. If $\mathbb{S}_{\mathcal{H}_0} = \mathbb{M}_T$,
where $\mathcal{H}_0$ is a finite dimensional subspace of $\mathcal{H}$
and $\|T\|_{{\mathcal{H}_0}^\perp}< \|T\|$. Then the following
statements are equivalent:
\begin{itemize}
\item[(i)] $dw(T) = \sqrt{w^2(T) + \|T\|^4}$.
\item[(ii)] There exists a unit vector $x\in \mathcal{H}_0$
such that $|\langle Tx, x\rangle| = \|T\|$.
\end{itemize}
\end{corollary}

\textbf{Acknowledgement.}
This research is supported by a grant from the Iran National Science Foundation (INSF- No. 95013683).

\bibliographystyle{amsplain}

\begin{thebibliography}{99}

\bibitem{B.S} R. Bhatia and P. \v{S}emrl,
\textit{Orthogonality of matrices and some distance problems},
Linear Algebra Appl. \textbf{287}(1-3) (1999), 77--85.

\bibitem{B.G} T. Bhattacharyya and P. Grover,
\textit{Characterization of Birkhoff-–James orthogonality},
J. Math. Anal. Appl. \textbf{407} (2013), no. 2, 350-–358.

\bibitem{C.G.L} M. T. Chien, H. L. Gau, C.-K. Li, M. C. Tsai, and K. Z. Wang,
\textit{Product of operators and numerical range}, Linear Multilinear Algebra,
\textbf{64} (2016), no. 1, 58--67.

\bibitem{Da} C. Davis,
\textit{The shell of a Hilbert-space operator},
Acta Sci. Math. (Szeged) \textbf{29} (1968), 69--86.

\bibitem{D.1} S. S. Dragomir,
\textit{A survey of some recent inequalities
for the norm and numerical radius of operators in Hilbert spaces},
Banach J. Math. Anal. \textbf{1} (2007), no. 2, 154--175.

\bibitem{G} P. Grover,
\textit{Orthogonality of matrices in the Ky Fan $k$-norms},
Linear Multilinear Algebra, \textbf{65} (2017), no. 3, 496--509.

\bibitem{G.R} K. E. Gustafson and D. K. M. Rao,
\textit{Numerical range. The field of values of linear operators and matrices},
Universitext. Springer-Verlag, New York, 1997.

\bibitem{K} F. Kittaneh,
\textit{Numerical radius inequalities for Hilbert space operators},
Studia Math. \textbf{168} (2005), no. 1, 73--80.

\bibitem{Li} C.-K. Li,
\textit{Linear operators preserving the numerical radius of matrices},
Proc. Amer. Math. Soc. \textbf{99} (1987), 601-–608.

\bibitem{L.Po} C.-K. Li, Y. T. Poon, and N. S. Sze,
\textit{Davis–-Wielandt shells of operators},
Oper. Matrices. \textbf{2} (2008), 341-–355.

\bibitem{M.Z} M. S. Moslehian and A. Zamani,
\textit{Characterizations of operator Birkhoff--James orthogonality},
Canad. Math. Bull. \textbf{60} (2017), no. 4, 816--829.

\bibitem{P.S.G} K. Paul, D. Sain, and P. Ghosh,
\textit{Birkhoff--James orthogonality and smoothness of bounded linear operators},
Linear Algebra Appl. \textbf{506} (2016), 551--563.

\bibitem{T} A. Turn\v{s}ek,
\textit{A remark on orthogonality and symmetry of operators in $B(H)$},
Linear Algebra Appl. \textbf{535} (2017), 141--150.

\bibitem{Wi} H. Wielandt,
\textit{On eigenvalues of sums of normal matrices},
Pacific J. Math. \textbf{5} (1955), 633--638.

\bibitem{W} P. W\'{o}jcik,
\textit{Norm-parallelism in classical $M$-ideals},
Indag. Math. (N.S.) \textbf{28} (2017), no. 2, 287--293.

\bibitem{Y} T. Yamazaki,
\textit{On upper and lower bounds of the numerical radius and an equality condition},
Studia Math. \textbf{178} (2007), no. 1, 83--89.

\bibitem{Z.1} A. Zamani,
\textit{The operator-valued parallelism},
Linear Algebra Appl. \textbf{505} (2016), 282--295.

\bibitem{Z.2} A. Zamani,
\textit{Some lower bounds for the numerical radius of Hilbert space operators},
Adv. Oper. Theory \textbf{2} (2017), no. 2, 98-–107


\bibitem{Z.M.1} A. Zamani and M. S. Moslehian,
\textit{Exact and approximate operator parallelism},
Canad. Math. Bull. \textbf{58}(1) (2015), 207--224.

\bibitem{Z.M.2} A. Zamani and M. S. Moslehian,
\textit{Norm-parallelism in the geometry of Hilbert $C^*$-modules},
Indag. Math. (N.S.) \textbf{27} (2016), no. 1, 266--281.

\end{thebibliography}

\end{document}